\documentclass[11pt]{article}
\usepackage{amsmath}
\usepackage{amssymb}
\usepackage{amsthm}
\usepackage[usenames]{color}
\usepackage{amscd}
\usepackage{indentfirst}

\usepackage[colorlinks=true,linkcolor=blue,filecolor=red,
citecolor=webgreen]{hyperref}
\definecolor{webgreen}{rgb}{0,.5,0}

\def\N{{\Bbb N}}

\def\C{{\Bbb C}}

\def\1{{\bf 1}}

\def\eps{{\varepsilon}}

\numberwithin{equation}{section}

\newtheorem{theorem}{Theorem}[section]

\newtheorem{remark}[theorem]{Remark}

\begin{document}

\title{{\bf On certain sums concerning the gcd's and lcm's of $k$ positive integers}}
\author{Titus Hilberdink, Florian Luca, and L\'aszl\'o T\'oth}
\date{}
\maketitle

\centerline{International Journal of Number Theory, Vol. 16, No. 1 (2020) 77--90}

\begin{abstract}
We use elementary arguments to prove results on the order of magnitude of certain sums concerning the gcd's and lcm's
of $k$ positive integers, where $k\ge 2$ is fixed. We refine and generalize an asymptotic formula of Bordell\`{e}s (2007), and extend certain
related results of Hilberdink and T\'oth (2016). We also formulate some conjectures and open problems.
\end{abstract}



\section{Introduction}

Consider the gcd-sum function
\begin{equation*}
G(n):=\sum_{k=1}^n (k,n)=\sum_{d\mid n} d\varphi(n/d)  \qquad (n\in \N),
\end{equation*}
where $\varphi(n)$ is Euler's totient function. The function $G(n)$ is multiplicative and the asymptotic formula
\begin{equation} \label{G_asymp}
\sum_{n\le x} G(n)=\frac{x^2}{2\zeta(2)} \left(\log x +2\gamma-\frac1{2}-\frac{\zeta'(2)}{\zeta(2)}\right) + {\cal
O}(x^{1+\theta+\varepsilon}),
\end{equation}
holds for every $\varepsilon>0$, where $\gamma$ is Euler's constant, and $\theta$ is the exponent appearing in
Dirichlet's divisor problem. See the survey paper \cite{Tot2010} by the third author.

The function
\begin{equation*}
G^{(-1)}(n):=\sum_{k=1}^n \frac1{(k,n)} =\sum_{d\mid n} \frac{\varphi(n/d)}{d}   \qquad (n\in \N),
\end{equation*}
is also multiplicative. Bordell\`{e}s \cite[Th.\ 5.1]{Bor2007} deduced that
\begin{equation} \label{G_minus_1}
\sum_{n\le x} G^{(-1)}(n) = \frac{\zeta(3)}{2\zeta(2)} x^2 + O\left(x(\log x)^{2/3}(\log \log x)^{4/3}\right).
\end{equation}

The error term of estimate \eqref{G_minus_1} comes from the classical result of Walfisz \cite[Satz 1, p.\ 144]{Wal1963},
\begin{equation} \label{Walfisz}
R(x):= \sum_{n\le x} \varphi(n) - \frac1{2\zeta(2)} x^2 = O\left(x(\log x)^{2/3}(\log \log x)^{4/3}\right).
\end{equation}

We remark that recently \eqref{Walfisz} was improved by Liu \cite{Liu2016} into
\begin{equation} \label{Liu}
R(x)= O\left(x(\log x)^{2/3}(\log \log x)^{1/3}\right),
\end{equation}
therefore, this serves as the remainder of \eqref{G_minus_1}. Also see the preprint by Suzuki \cite{Suz2018}.

The lcm-sum function
\begin{equation*}
L(n): =\sum_{k=1}^n [k,n] = \frac{n}{2} \left(1+\sum_{d\mid n}
d\varphi(d)\right) \qquad (n\in \N).
\end{equation*}
was investigated by Bordell\`{e}s \cite{Bor2007}, Ikeda and Matsuoka \cite{IkeMat2014}, and others.
The function $L(n)$ is not multiplicative and one has, see \cite[Th.\ 6.3]{Bor2007},
\begin{equation} \label{asympt_L}
\sum_{n\le x} L(n) = \frac{\zeta(3)}{8\zeta(2)} x^4 + O\left(x^3 (\log x)^{2/3}(\log \log x)^{4/3} \right).
\end{equation}

By using \eqref{Liu}, the exponent of the $\log \log x$ factor in the error of \eqref{asympt_L} can be improved into $1/3$.

Now let
\begin{equation*}
L^{(-1)}(n): =\sum_{k=1}^n \frac1{[k,n]} \qquad (n\in \N).
\end{equation*}

Bordell\`{e}s \cite[Th.\, 7.1]{Bor2007} proved that
\begin{equation} \label{lcm_recipr_two_var}
\sum_{n\le x} L^{(-1)}(n) = \frac1{\pi^2} (\log x)^3+ A (\log x)^2 + O(\log x),
\end{equation}
with an explicitly given constant $A$.

By the general identity
\begin{equation*} \label{f_symm}
\sum_{m,n\le x} \psi(m,n) = 2 \sum_{n\le x} \sum_{m=1}^n \psi(m,n) -\sum_{n\le x} \psi(n,n),
\end{equation*}
valid for any function $\psi:\N^2 \to \C$ , which is symmetric in the variables, \eqref{G_asymp}, \eqref{G_minus_1}, \eqref{asympt_L}
and \eqref{lcm_recipr_two_var}, together with the remark on \eqref{Liu} lead to the asymptotic formulas
\begin{equation} \label{gcd_m_n}
\sum_{m,n\le x} (m,n)= \frac{x^2}{\zeta(2)}\left(\log x+ 2\gamma
-\frac1{2}-\frac{\zeta(2)}{2}- \frac{\zeta'(2)}{\zeta(2)} \right) +
O\left(x^{1+\theta+\varepsilon}\right),
\end{equation}
\begin{equation} \label{gcd_recipr_m_n}
\sum_{m,n\le x} \frac1{(m,n)} = \frac{\zeta(3)}{\zeta(2)} x^2 + O\left(x(\log x)^{2/3}(\log \log x)^{1/3} \right),
\end{equation}
\begin{equation} \label{lcm_m_n}
\sum_{m,n\le x} [m,n] = \frac{\zeta(3)}{4\zeta(2)} x^4 + O\left(x^3(\log x)^{2/3}(\log \log x)^{1/3} \right),
\end{equation}
and
\begin{equation} \label{lcm_recipr_m_n}
\sum_{m,n\le x} \frac1{[m,n]} = \frac{2}{\pi^2} (\log x)^3+ A_1 (\log x)^2 + O(\log x),
\end{equation}
respectively, where $A_1=2A$.

It is easy to generalize \eqref{gcd_m_n} and \eqref{gcd_recipr_m_n} for sums with $k$ variables by using the general identity
\begin{equation*}
\sum_{n_1,\ldots,n_k\le x} f((n_1,\ldots,n_k)) = \sum_{d\le x}(\mu*f)(d)  \lfloor x/d \rfloor^k,
\end{equation*}
where $f$ is an arbitrary arithmetic function, $\mu$ is the M\"obius function and $*$ stands for the Dirichlet convolution of arithmetic
functions. For example, we have the next result: For any $k\ge 3$,
\begin{equation*}
\sum_{n_1,\ldots, n_k\le x} \frac1{(n_1,\ldots, n_k)} = \frac{\zeta(k+1)}{\zeta(k)} x^k + O\left(x^{k-1}\right).
\end{equation*}

However, it is more difficult to derive asymptotic formulas for similar sums involving the lcm $[n_1,\ldots,n_k]$.
As corollaries of more general results concerning a large class of functions $f$, the first
and third authors \cite[Cor\, 1]{HilTot2016} proved that for any $k\ge 3$ and any real number $r>-1$,
\begin{equation} \label{lcm_k_dim}
\sum_{n_1,\ldots, n_k\le x} [n_1,\ldots, n_k]^r = A_{r,k} x^{k(r+1)} + O\left(x^{k(r+1)-\frac1{2}\min (r+1,1)+\varepsilon}\right)
\end{equation}
and
\begin{equation*}
\sum_{n_1,\ldots, n_k\le x} \left(\frac{[n_1,\ldots, n_k]}{n_1\cdots n_k}\right)^r = A_{r,k} x^k + O\left(x^{k-\frac1{2}\min (r+1,1)+\varepsilon}\right),
\end{equation*}
where $A_{k,r}$ are explicitly given constants. Here, \eqref{lcm_k_dim} is the $k$ dimensional generalization of \eqref{lcm_m_n}. Furthermore,
\cite[Cor\, 2]{HilTot2016} shows that for any $k\ge 3$ and any real number $r>0$,
\begin{equation*}
\sum_{n_1,\ldots, n_k\le x} \left(\frac{[n_1,\ldots, n_k]}{(n_1,\ldots,n_k)}\right)^r = B_{r,k} x^{k(r+1)} + O\left(x^{k(r+1)-\frac1{2}+\varepsilon}\right),
\end{equation*}
with explicitly given constants $B_{k,r}$. The proofs use the fact that $(n_1,\ldots,n_k)$ and $[n_1,\ldots,n_k]$ are multiplicative functions of $k$ variables
and the associated  multiple Dirichlet series factor over the primes into Euler products. The proofs given in \cite{HilTot2016} cannot be applied in
the case $r=-1$.

It is the goal of the present paper to investigate the order of magnitude of the sums
\begin{equation} \label{S_k_x}
S_k(x):= \sum_{n_1,\ldots,n_k\le x} \frac1{[n_1,\ldots,n_k]},
\end{equation}
\begin{equation} \label{T_k_x}
T_k(x):= \sum_{n_1,\ldots, n_k\le x} \frac{(n_1,\ldots, n_k)}{[n_1,\ldots,n_k]},
\end{equation}
\begin{equation} \label{U_k_x}
U_k(x):= \sum_{\substack{n_1,\ldots, n_k\le x\\ (n_1,\ldots,n_k)=1}} \frac1{[n_1,\ldots,n_k]},
\end{equation}
\begin{equation} \label{V_k_x}
V_k(x):= \sum_{n_1,\ldots, n_k\le x} \frac{n_1\cdots n_k}{[n_1,\ldots,n_k]},
\end{equation}
where $k\ge 2$ is fixed, by using elementary arguments. Theorem \ref{Th_L_m1_k_2}, concerning the sum $S_2(x)$,
refines formulas \eqref{lcm_recipr_two_var} and \eqref{lcm_recipr_m_n} of Bordell\`{e}s \cite{Bor2007}.
Theorems \ref{Th_sum_S_k} and \ref{Th_sum_U_k} give the exact order of magnitude of the sums
$S_k(x)$ and $U_k(x)$, respectively, for $k\ge 3$. Theorem \ref{Th_sum_V_k} concerns the sums $V_k(x)$, while Theorem \ref{Th_sum_T_k}
provides an asymptotic formula with remainder term for $T_k(x)$, for any fixed $k\ge 2$. Some conjectures and open problems are formulated as well.

\section{The sums $S_k(x)$}

First consider the sums $S_k(x)$ defined by \eqref{S_k_x}. In the case $k=2$ we use Dirichlet's hyperbola method to prove the next
result, which improves formulas \eqref{lcm_recipr_two_var} and \eqref{lcm_recipr_m_n}.

\begin{theorem}\label{Th_L_m1_k_2}
\begin{equation} \label{L_minus_1}
\sum_{n\le x} L^{(-1)}(n) = \frac1{\pi^2} (\log x)^3+ A (\log x)^2 + B \log x + C +O\left(x^{-1/2} (\log x)^2\right),
\end{equation}
that is,
\begin{equation*}
\sum_{m,n\le x} \frac1{[m,n]}  = \frac{2}{\pi^2} (\log x)^3+ A_1 (\log x)^2 + B_1 \log x + C_1 +O\left(x^{-1/2} (\log x)^2\right),
\end{equation*}
where the constants $A,B,C$ can be explicitly computed, and $A_1=2A$, $B_1=2B-1$, $C_1=C -\gamma$.
\end{theorem}

\begin{proof} We have
\begin{equation} \label{n_L}
L^{(-1)}(n) =\sum_{k=1}^n \frac{(k,n)}{kn} = \frac{1}{n} \sum_{d\mid n} d \sum_{\substack{k=1\\ (k,n)=d}}^n \frac1{k}
= \frac{1}{n} \sum_{d\mid n} \sum_{\substack{t=1\\ (t,n/d)=1}}^{n/d} \frac1{t}
= \frac{1}{n}\sum_{d\mid n} h(d),
\end{equation}
where
\begin{equation*}
h(n):=\sum_{\substack{m=1\\ (m,n)=1}}^n \frac1{m} =  \sum_{m=1}^n \frac1{m} \sum_{d\mid (m,n)} \mu(d) = \sum_{d\mid n} \frac{\mu(d)}{d}
\sum_{j=1}^{n/d} \frac1{j}
\end{equation*}
\begin{equation*}
= \sum_{d\mid n} \frac{\mu(d)}{d} \left(\log \frac{n}{d}+\gamma+O\left(\frac{d}{n}\right) \right)
= \sum_{d\mid n} \frac{\mu(d)}{d} \log \frac{n}{d} + \gamma \frac{\varphi(n)}{n} + O\left(\frac{2^{\omega(n)}}{n}\right).
\end{equation*}

Hence,
\begin{equation*}
H(x):= \sum_{n\le x} h(n) =  \sum_{d\le x} \frac{\mu(d)}{d} \sum_{m\le x/d} \log m + \gamma \sum_{n\le x} \frac{\varphi(n)}{n} +
O\left(\sum_{n\le x} \frac{2^{\omega(n)}}{n}\right).
\end{equation*}

By using the known estimates
\begin{gather*}
\sum_{n\le x} \log n = x\log x -x + O(\log x), \\
\sum_{n\le x} \frac{\varphi(n)}{n}=\frac{6}{\pi^2}x +O(\log x),\\
\sum_{n\le x} \frac{2^{\omega(n)}}{n} = O((\log x)^2),
\end{gather*}
we deduce that
\begin{equation*}
H(x)= (x\log x - x) \sum_{d\le x} \frac{\mu(d)}{d^2} - x \sum_{d\le x} \frac{\mu(d)\log d}{d^2} +\frac{6}{\pi^2}\gamma x
+O((\log x)^2)
\end{equation*}
\begin{equation} \label{H_x}
= \frac{6}{\pi^2} (x\log x +cx) +O((\log x)^2),
\end{equation}
with a certain constant $c$. Let $\1(n)=1$ ($n\in \N$), and let $*$ denote the Dirichlet convolution. By Dirichlet's hyperbola method,
\begin{equation*}
\sum_{n\le x} (\1*h)(n)  =  \sum_{n\le\sqrt{x}} \left(H(x/n) + h(n) \lfloor x/n \rfloor \right) -\lfloor \sqrt{x} \rfloor H(\sqrt{x})
\end{equation*}
\begin{equation*}
= \sum_{n\le\sqrt{x}} H(x/n) + x\sum_{n\le\sqrt{x}}\frac{h(n)}{n} - \sqrt{x} H(\sqrt{x})+O(H(\sqrt{x})).
\end{equation*}

By partial summation,
\begin{equation*}
x\sum_{n\le\sqrt{x}}\frac{h(n)}{n} = \sqrt{x} H(\sqrt{x}) + x\int_1^{\sqrt{x}} \frac{H(t)}{t^2}\, dt,
\end{equation*}
and  using \eqref{H_x} we deduce
\begin{gather*}
\sum_{n\le x} (\1*h)(n) =  \frac{6}{\pi^2} \sum_{n\le\sqrt{x}} \left( \frac{x}{n}\log \left(\frac{x}{n}\right)+c \left(\frac{x}{n}\right)\right)
+ \frac{6x}{\pi^2} \int_1^{\sqrt{x}} \left( \frac{\log t}{t} +c \right)\frac{dt}{t}  + O\left(\sqrt{x}(\log x)^2\right)  \\
= x\left(\frac{3}{\pi^2}(\log x)^2 + a\log x + b\right) +O(\sqrt{x}(\log x)^2),
\end{gather*}
for some constants $a,b$, which can be explicitly calculated.

Here $(\1*h)(n)=n L^{(-1)}(n)$, according to \eqref{n_L}, and  we obtain \eqref{L_minus_1} by partial summation. \end{proof}

It is more difficult to handle the sums $S_k(x)$ in the case $k\ge 3$.  We will apply the following general result proved by the second
and third authors \cite{LucTot2017}, using elementary arguments.

\begin{theorem} \textup{(\cite{LucTot2017})} \label{Th_Luca_Toth}
Let $k$ be a positive integer and let $f:\N \to \C$ be a multiplicative function satisfying the following properties:

(i) $f(p)= k$ for every prime $p$,

(ii) $f(p^{\nu})=\nu^{O(1)}$ for every prime $p$ and every integer $\nu \ge 2$, where the constant implied by the $O$ symbol is uniform in $p$.

Then
\begin{equation*}
\sum_{n\le x} \frac{f(n)}{n}= \frac1{k!} C_{f} (\log x)^k + D_{f} (\log x)^{k-1} + O \left((\log x)^{k-2}\right),
\end{equation*}
where $C_f$ and $D_f$ are constants,
$$
C_{f}= \prod_p \left(1-\frac{1}{p}\right)^k \left(\sum_{\nu =0}^{\infty} \frac{f(p^{\nu})}{p^{\nu}} \right).
$$
\end{theorem}

We have the following result.

\begin{theorem} \label{Th_sum_S_k} Let $k\ge 3$ be a fixed integer. Then
\begin{equation*}
S_k(x) \asymp (\log x)^{2^k-1} \qquad \text{ as $x\to \infty$.}
\end{equation*}
\end{theorem}

\begin{proof} Since $[n_1,\ldots, n_k]\le n_1\cdots n_k\le x^k$, we can write
\begin{equation} \label{S_k}
S_k(x) = \sum_{n\le x^k} \frac{1}{n} \sum_{\substack{n_1,\ldots, n_k\le x \\ [n_1, \ldots, n_k]=n}} 1
\end{equation}

Let
\begin{equation*} \label{def_a_k_n}
a_k(n):= \sum_{\substack{n_1,\ldots, n_k\in \N \\ [n_1,\ldots,n_k]=n}} 1.
\end{equation*}

Now if $n\le x$, then the inner sum in \eqref{S_k} is just $a_k(n)$ (since $n\le x$ forces $n_1,\ldots,n_k\le x$), while in any case it is at
most $a_k(n)$. Thus
\begin{equation} \label{S_k_a_k_n}
\sum_{n\le x} \frac{a_k(n)}{n} \le S_k(x)\le \sum_{n\le x^k} \frac{a_k(n)}{n}.
\end{equation}

To see the properties of the function $a_k(n)$ write
\begin{equation*}
\sum_{d\mid n} a_k(d) = \sum_{d\mid n} \sum_{[n_1,\ldots,n_k]=d} 1 = \sum_{[n_1,\ldots,n_k]\mid n} 1
=  \sum_{n_1\mid n,\ldots,n_k\mid n} 1 = \tau(n)^k.
\end{equation*}

Therefore, by M\"obius inversion, we have $a_k=\mu*\tau^k$. This shows that $a_k(n)$ is multiplicative and its values at the prime powers $p^\nu$ are
given by $a_k(p^\nu)=(\nu+1)^k-\nu^k$ ($\nu \ge 1$). In particular, $a_k(p)=2^k-1$.

Applying Theorem \ref{Th_Luca_Toth} for the function $f(n)=a_k(n)$, with $2^k-1$ instead of $k$, we get that
\begin{equation} \label{c_log}
\sum_{n\le x} \frac{a_k(n)}{n} \sim \alpha_k (\log x)^{2^k-1} \qquad \text{ as $x\to \infty$},
\end{equation}
for some constant $\alpha_k$. Now, from \eqref{S_k_a_k_n} and \eqref{c_log} the result follows.
\end{proof}

\begin{remark} {\rm It is natural to expect that $S_k(x) \sim c_k (\log x)^{2^k-1}$ as $x\to \infty$, with a certain constant $c_k$. In fact,
in view of Theorem \ref{Th_L_m1_k_2}, the plausible conjecture is that
\begin{equation} \label{conj_L_minus_1}
S_k(x) = P_{2^k-1}(\log x) + O(x^{-r}),
\end{equation}
where $P_{2^k-1}(t)$ is a polynomial in $t$ of degree $2^k-1$ and $r$ is a positive real number. We pose as an open problem to find the
constants $c_k$ and to prove \eqref{conj_L_minus_1}.
}
\end{remark}

\section{The sums $U_k(x)$}

Next consider the sums $U_k(x)$ defined by \eqref{U_k_x}. In the case $k=2$,
\begin{equation*}
U_2(x)\sim \frac{6}{\pi^2} (\log x)^2\qquad \text{ as $x\to \infty$},
\end{equation*}
and it is not difficult to deduce a more precise asymptotic formula.

We have the following general result.

\begin{theorem} \label{Th_sum_U_k} Let $k\ge 3$ be a fixed integer. Then
\begin{equation*}
U_k(x) \asymp (\log x)^{2^k-2} \qquad \text{ as $x\to \infty$.}
\end{equation*}
\end{theorem}

\begin{proof} Similar to the proof of Theorem \ref{Th_sum_S_k}. We have
\begin{equation} \label{U_k}
U_k(x)= \sum_{\substack{n_1,\ldots,n_k\le x\\ (n_1,\ldots,n_k)=1}} \frac1{[n_1,\ldots,n_k]}
= \sum_{n\le x^k} \frac{1}{n} \sum_{\substack{n_1,\ldots, n_k\le x \\ [n_1, \ldots, n_k]=n\\(n_1,\ldots,n_k)=1}} 1.
\end{equation}

Let
\begin{equation*}
b_k(n) = \sum_{\substack{n_1,\ldots, n_k\in \N \\ [n_1,\ldots,n_k]=n\\(n_1,\ldots,n_k)=1}} 1.
\end{equation*}

Now if $n\le x$, then the inner sum in \eqref{U_k} is exactly $b_k(n)$, while in any case it is at
most $b_k(n)$. Thus
\begin{equation} \label{b_k_n}
\sum_{n\le x} \frac{b_k(n)}{n} \le U_k(x)\le \sum_{n\le x^k} \frac{b_k(n)}{n}.
\end{equation}

Write
\begin{equation*}
\sum_{d\mid n} b_k(d) = \sum_{d\mid n} \sum_{\substack{[n_1,\ldots,n_k]=d\\(n_1,\ldots,n_k)=1}} 1 = \sum_{\substack{[n_1,\ldots,n_k]
\mid n \\(n_1,\ldots,n_k)=1}} 1
\end{equation*}
\begin{equation*}
=  \sum_{n_1\mid n,\ldots,n_k \mid n} \sum_{\delta \mid (n_1,\ldots,n_k)}
\mu(\delta) = \sum_{\delta a_1b_1=n,\ldots,\delta a_kb_k=n} \mu(\delta)
\end{equation*}
\begin{equation*}
= \sum_{\delta t=n} \mu(\delta) \sum_{a_1b_1=t} 1 \cdots \sum_{a_kb_k=t} 1 = \sum_{\delta t=n} \mu(\delta) \tau(t)^k.
\end{equation*}

Therefore, by M\"obius inversion $b_k=\mu* \mu*\tau^k$. This shows that $b_k(n)$ is multiplicative and its values at the prime powers $p^\nu$ are
given by $b_k(p^\nu)=(\nu+1)^k-2\nu^k+(\nu-1)^k$ ($\nu \ge 1$). In particular, $b_k(p)=2^k-2$.

Applying now Theorem \ref{Th_Luca_Toth} for the function $f(n)=b_k(n)$, with $2^k-2$ instead of $k$,
we  deduce that
\begin{equation} \label{ccc_log}
\sum_{n\le x} \frac{b_k(n)}{n} \sim \alpha_k' (\log x)^{2^k-2} \qquad \text{ as $x\to \infty$}
\end{equation}
for some constant $\alpha_k'$. Now, from \eqref{b_k_n} and \eqref{ccc_log} we have
$U_k(x) \asymp  (\log x)^{2^k-2}$.
\end{proof}

\begin{remark} {\rm We conjecture that $U_k(x) \sim d_k (\log x)^{2^k-2}$ as $x\to \infty$, with a certain constant $d_k$.
The sums $S_k(x)$ and $U_k(x)$ are strongly related. Namely, by grouping the terms according to the values $(n_1,\ldots,n_k)=d$ one
obtains
\begin{equation} \label{S_U}
S_k(x)= \sum_{d\le x} \frac1{d} U_k(x/d),
\end{equation}
and conversely,
\begin{equation} \label{U_S}
U_k(x)= \sum_{d\le x} \frac{\mu(d)}{d} S_k(x/d).
\end{equation}

If $U_k(x) \sim d_k (\log x)^{2^k-2}$ holds, then by \eqref{S_U} it follows that $S_k(x) \sim \frac{d_k}{2^k-1} (\log x)^{2^k-1}$.
Conversely, assume that the asymptotic formula \eqref{conj_L_minus_1} is true, where $c_k$ is the leading coefficient of the
polynomial $P_{2^k-1}(t)$. Then \eqref{U_S}, together with the well known results
\begin{equation*}
\sum_{n\le x} \frac{\mu(n)}{n} = O((\log x)^{-1}), \qquad \sum_{n=1}^{\infty} \frac{\mu(n)\log n}{n} = -1,
\end{equation*}
and Shapiro's estimates \cite[Th.\ 4.1]{Sha1950}
\begin{equation*}
\sum_{n \le x} \frac{\mu(n)}{n} \left(\log \left(\frac{x}{n}\right)\right)^m = m (\log x)^{m-1} + \sum_{i=1}^{m-2} c_j^{(m)} (\log x )^j +O(1),
\end{equation*}
valid for any integer $m\ge 2$, where $c_i^{(m)}$ are constants, imply that
\begin{equation*}
U_k(x)= (2^k-1) c_k(\log x)^{2^k-2}+ b_{2^k-3}(\log x)^{2^k-3}+\cdots + b_1 \log x+ O(1),
\end{equation*}
with some constants $b_i$.
}
\end{remark}

\section{The sums $V_k(x)$}

The sums $V_k(x)$ defined by \eqref{V_k_x} are sums of integers. In the case $k=2$ we have,
according to \eqref{gcd_m_n},
\begin{equation} \label{V_2}
V_2(x) = \sum_{m,n\le x} (m,n) \sim \frac{6}{\pi^2}x^2\log x.
\end{equation}

\begin{theorem} \label{Th_sum_V_k} Let $k\ge 3$ be a fixed integer. Then
\begin{equation*}
x^k\ll V_k(x) \ll x^k (\log x)^{2^k-2}\qquad \text{ as $x\to \infty$.}
\end{equation*}
\end{theorem}

\begin{proof} The lower bound is trivial by $n_1\cdots n_k\ge [n_1,\ldots,n_k]$. Also, by grouping the terms according to
the values $(n_1,\ldots,n_k)=d$, and by denoting $M=\max(m_1,\ldots,m_k)$ we have
\begin{equation*}
V_k(x)= \sum_{\substack{dm_1,\ldots, dm_k\le x \\ (m_1,\ldots, m_k)=1}} \frac{dm_1\cdots dm_k}{[dm_1,\ldots,dm_k]}
= \sum_{\substack{m_1,\ldots, m_k\le x \\ (m_1,\ldots, m_k)=1}} \frac{m_1\cdots m_k}{[m_1,\ldots,m_k]} \sum_{d\le x/M} d^{k-1}
\end{equation*}
\begin{equation*}
\ll x^k \sum_{\substack{m_1,\ldots, m_k\le x \\ (m_1,\ldots, m_k)=1}} \frac{m_1\cdots m_k}{[m_1,\ldots,m_k] M^k }
\le x^k \sum_{\substack{m_1,\ldots, m_k\le x \\ (m_1,\ldots, m_k)=1}} \frac1{[m_1,\ldots,m_k]} = x^k U_k(x),
\end{equation*}
and the upper bound follows from Theorem \ref{Th_sum_U_k}.
\end{proof}

\begin{remark} {\rm We conjecture that $V_k(x) \sim \lambda_k x^k (\log x)^{2^k-k-1}$ as $x\to \infty$, with a certain constant $\lambda_k$, in
accordance with \eqref{V_2} for the case $k=2$. We pose as another open problem to prove this and to find the constants $\lambda_k$.
}
\end{remark}

\section{The sums $T_k(x)$}

Finally, we investigate the sums $T_k(x)$ defined by \eqref{T_k_x} and establish an asymptotic formula with remainder term for it.
We give a short direct proof in the case $k=2$. Then for any fixed $k\ge 2$ we use multiple Dirichlet series to get the result.

Let
\begin{equation}
F(n): =\sum_{k=1}^n \frac{(k,n)}{[k,n]} \qquad (n\in \N).
\end{equation}

\begin{theorem} \label{Th_sum_T_2}
\begin{equation}
\sum_{n\le x} F(n) = 2 x + O\left((\log x)^2\right),
\end{equation}
that is,
\begin{equation*}
\sum_{m,n\le x} \frac{(m,n)}{[m,n]} = 3 x + O\left((\log x)^2\right).
\end{equation*}
\end{theorem}

\begin{proof} Let $\phi_2(n)=\sum_{d\mid n} d^2\mu(n/d)$ be the Jordan function of order $2$. We have
\begin{equation*}
F(n) = \sum_{k=1}^n \frac{(k,n)^2}{kn} = \frac1{n} \sum_{k=1}^n \frac1{k} \sum_{d\mid (k,n)} \phi_2(d)
= \frac1{n} \sum_{d\mid n} \phi_2(d) \sum_{\substack{k=1\\ d\mid k}}^n \frac1{k}
\end{equation*}
\begin{equation*}
= \frac1{n} \sum_{d\mid n} \frac{\phi_2(d)}{d} \sum_{j=1}^{n/d}  \frac1{j}=
\frac1{n} \sum_{d\mid n} \frac{\phi_2(d)}{d} H_{n/d},
\end{equation*}
where $H_m= \sum_{j=1}^m 1/j$ is the harmonic sum. Therefore, using that
\begin{equation*}
\sum_{n\le x} \frac{\phi_2(n)}{n^2} =\frac{x}{\zeta(3)} + O(1),
\end{equation*}
we deduce
\begin{equation*}
\sum_{n\le x} F(n) = \sum_{dm\le x} \frac{\phi_2(d)}{d^2m} H_m = \sum_{m\le x} \frac{H_m}{m} \sum_{d\le x/m} \frac{\phi_2(d)}{d^2}
\end{equation*}
\begin{equation*}
=\sum_{m\le x} \frac{H_m}{m} \left(\frac{x}{\zeta(3)m} +O(1)\right)
= \frac{x}{\zeta(3)} \sum_{m\le x} \frac{H_m}{m^2} + O\left(\sum_{m\le x} \frac{H_m}{m}\right)
\end{equation*}
\begin{equation*}
= \frac{x}{\zeta(3)} \sum_{m=1}^{\infty} \frac{H_m}{m^2} + O(x\sum_{m>x} \frac{H_m}{m^2})+ O\left(\sum_{m\le x} \frac{H_m}{m}\right)
\end{equation*}
\begin{equation*}
= \frac{x}{\zeta(3)}\cdot 2\zeta(3) + O\left(x\sum_{m>x} \frac{\log m}{m^2}\right)+ O\left(\sum_{m\le x} \frac{\log m}{m}\right)
= 2x + O((\log x)^2),
\end{equation*}
by using that
\begin{equation} \label{form_Euler}
\sum_{n=1}^{\infty} \frac{H_n}{n^2}= 2 \zeta(3),
\end{equation}
which is Euler's result.
\end{proof}

\begin{theorem} \label{Th_sum_T_k} If $k\ge 2$, then
\begin{equation*}
T_k(x) = \beta_k x+ O\left((\log x)^{2^k-2}\right),
\end{equation*}
where
\begin{align*}
\beta_k:= \sum_{\substack{n_1,\ldots,n_k=1\\ (n_1,\ldots,n_k)=1}}^{\infty} \frac1{[n_1,\ldots,n_k]\max(n_1,\ldots,n_k)}
=\frac1{\zeta(2)} \sum_{n_1,\ldots,n_k=1}^{\infty} \frac1{[n_1,\ldots,n_k]\max(n_1,\ldots,n_k)}.
\end{align*}
\end{theorem}

\begin{proof} By grouping the terms according to $(n_1,\ldots,n_k)=d$, where
$n_j=d m_j$ ($1\le j\le k$), $(m_1,\ldots,m_k)=1$, we have
\begin{equation*}
T_k(x)= \sum_{\substack{dm_1,\ldots,dm_k\le x\\ (m_1,\ldots,m_k)=1}} \frac{d}{[dm_1,\ldots,dm_k]}
= \sum_{\substack{dm_1,\ldots,dm_k\le x\\ (m_1,\ldots,m_k)=1}} \frac1{[m_1,\ldots,m_k]}
\end{equation*}
\begin{equation*}
=\sum_{\substack{m_1,\ldots,m_k\le x\\ (m_1,\ldots,m_k)=1}} \frac1{[m_1,\ldots,m_k]} \sum_{d\le x/M} 1
= \sum_{\substack{m_1,\ldots,m_k\le x\\ (m_1,\ldots,m_k)=1}} \frac{\lfloor x/M \rfloor}{[m_1,\ldots,m_k]},
\end{equation*}
where $M=\max(m_1,\ldots,m_k)$. Let
\begin{equation*}
h(n_1,\ldots,n_k):= \begin{cases} \frac1{[n_1,\ldots,n_k]},  & \text{ if $(n_1,\ldots,n_k)=1$},\\  0, & \text{ otherwise}.
\end{cases}
\end{equation*}

Hence,
\begin{equation} \label{T_k_sums}
T_k(x) = x \sum_{n_1,\ldots,n_k\le x} \frac{h(n_1,\ldots,n_k)}{\max(n_1,\ldots,n_k)}+ O\left( \sum_{n_1,\ldots,n_k\le x}
h(n_1,\ldots,n_k)\right)
\end{equation}
and we estimate the right-hand sums in turn. Here $h(n_1,\ldots,n_k)$ is a symmetric and multiplicative function of $k$ variables and
for prime powers $p^{\nu_1},\ldots,p^{\nu_k}$
($\nu_1,\ldots,\nu_k\ge 0$) one has
\begin{equation*}
h(p^{\nu_1},\ldots,p^{\nu_k})= \begin{cases} \frac1{p^{\max(\nu_1,\ldots,\nu_k)}},  & \text{ if $\min(\nu_1, \ldots, \nu_k)=0$},\\  0,
& \text{ otherwise}.
\end{cases}
\end{equation*}

Consider its Dirichlet series
\begin{equation*}
H(s_1,\ldots,s_k): = \sum_{n_1,\ldots,n_k=1}^{\infty} \frac{h(n_1,\ldots,n_k)}{n_1^{s_1}\cdots n_k^{s_k}} =
\prod_p \sum_{\substack{\nu_1,\ldots,\nu_k=0\\ \min(\nu_1, \ldots, \nu_k) =0}}^{\infty} \frac1{p^{\max(\nu_1,\ldots,\nu_k)+\nu_1s_1+\cdots +\nu_k s_k}}.
\end{equation*}

By grouping the terms according to the values of $r=\max(\nu_1,\ldots, \nu_k)$ we deduce
\begin{equation*}
H(s_1,\ldots,s_k)= \prod_p \frac1{p^r} \sum_{r=0}^{\infty} \sum_{\substack{\nu_1,\ldots,\nu_k=0\\ \max(\nu_1, \ldots, \nu_k)=r \\
\min(\nu_1, \ldots, \nu_k)=0}}^{\infty} \frac1{p^{\nu_1 s_1+\cdots +\nu_k s_k}},
\end{equation*}
which converges absolutely for $\Re s_j >0$ ($1\le j\le k$).

We shall need an estimate for $H_k(\eps,\ldots,\eps)$ for $\eps>0$ (small). We have
\begin{equation*}
H(\eps,\ldots,\eps)= \prod_p \left( 1+ \frac1{p} \sum_{j=1}^{k-1} \binom{k}{j}\frac1{p^{j\eps}}+ O\left(\frac1{p^2}\right)\right).
\end{equation*}

Therefore,
\begin{equation*}
\log H(\eps,\ldots, \eps) = \sum_p \frac{1}{p}\sum_{j=1}^{k-1} \binom{k}{j}\frac1{p^{j\eps}} +O(1) =
\sum_{j=1}^{k-1} \binom{k}{j} \sum_p \frac1{p^{1+j\eps}} +O(1).
\end{equation*}

But $\sum_p p^{-1-\eps} = \log\frac{1}{\eps}+O(1)$ as $\eps \to 0$. Thus,
\begin{equation} \label{bound_eps}
H(\eps,\ldots, \eps) = \exp \left( \sum_{j=1}^{k-1} \binom{k}{j} \log \frac{1}{\eps} +O(1)\right) \asymp
\left(\frac{1}{\eps}\right)^ {2^k-2}.
\end{equation}

Furthermore, for any $\eps>0$, we have
\begin{equation*}
\sum_{n_1,\ldots n_k\le x} h(n_1,\ldots,n_k)=  \sum_{n_1,\ldots n_k\le x} \frac{h(n_1,\ldots,n_k)}{(n_1\cdots n_k)^{\eps/k}}
(n_1\cdots n_k)^{\eps/k}
\end{equation*}
\begin{equation} \label{estimate_1}
\le x^{\eps} \sum_{n_1,\ldots n_k\le x} \frac{h(n_1,\ldots,n_k)}{(n_1\cdots n_k)^{\eps/k}}\le x^{\eps}
H(\eps/k,\ldots,\eps/k).
\end{equation}

Next, note that $\max(n_1,\ldots,n_k) \ge (n_1\cdots n_k)^{1/k}$, so that
\begin{equation*}
\sum_{n_1,\ldots,n_k\le x} \frac{h(n_1,\ldots,n_k)}{\max(n_1,\ldots,n_k)}\le \sum_{n_1,\ldots,n_k\le x} \frac{h(n_1,\ldots,n_k)}{(n_1\cdots n_k)^{1/k}}
\le H(\eps/k,\ldots,\eps/k),
\end{equation*}
which converges. Hence,
\begin{equation*}
\beta_k= \sum_{n_1,\ldots,n_k=1}^{\infty} \frac{h(n_1,\ldots,n_k)}{\max(n_1,\ldots,n_k)}
\end{equation*}
is finite and $\beta_k \le H(\eps/k,\ldots,\eps/k)$. Also,
\begin{equation*}
\beta_k - \sum_{n_1,\ldots,n_k\le x} \frac{h(n_1,\ldots,n_k)}{\max(n_1,\ldots,n_k)}
\end{equation*}
\begin{equation*}
= \sum_{\substack{n_1,\ldots,n_k\in \N \\ {\rm some\,} n_i> x}} \frac{h(n_1,\ldots,n_k)}{\max(n_1,\ldots,n_k)}
\end{equation*}
\begin{equation*}
\le k  \sum_{\substack{n_1\ge n_2,\ldots,n_k \\  n_1> x}} \frac{h(n_1,\ldots,n_k)}{n_1}
\le k  \sum_{\substack{n_1\ge n_2,\ldots,n_k \\  n_1> x}} \frac{h(n_1,\ldots,n_k)}{n_1^{1-\eps}(n_1n_2\cdots n_k)^{\eps/k}}
\end{equation*}
\begin{equation} \label{estimate_2}
\le \frac{k}{x^{1-\eps}}  \sum_{n_1,\ldots,n_k=1}^{\infty}  \frac{h(n_1,\ldots,n_k)}{(n_1\cdots n_k)^{\eps/k}}=kx^{\eps-1}H(\eps/k,\ldots,\eps/k).
\end{equation}

Hence, \eqref{T_k_sums} and the estimates \eqref{estimate_1}, \eqref{estimate_2} give
\begin{equation*}
T_k(x) =\beta_k x+ O\left(x^{\eps} H(\eps/k,\ldots,\eps/k)\right).
\end{equation*}

Now we choose $\eps=1/\log x$ and use the bound \eqref{bound_eps}. The proof is complete.
\end{proof}

\begin{remark} {\rm For $k=2$, Theorem \ref{Th_sum_T_k} recovers Theorem \ref{Th_sum_T_2}. Note that
\begin{equation*}
\beta_2= \frac1{\zeta(3)} \sum_{m,n=1}^{\infty} \frac1{mn\max(m,n)}= \frac{2}{\zeta(3)} \sum_{m=1}^{\infty} \frac1{m^2}
\sum_{n=1}^m \frac1{n}-1 = 3,
\end{equation*}
by Euler's result \eqref{form_Euler}. Is it possible to evaluate the constants $\beta_k$ for any $k\ge 2$?

The sums $T_k(x)$ and $U_k(x)$ are related by the formulas
\begin{equation*}
T_k(x)= \sum_{d\le x} U_k(x/d), \qquad
U_k(x)= \sum_{d\le x} \mu(d) T_k(x/d).
\end{equation*}
}
\end{remark}

\section*{Acknowledgments}
Part of the work in this paper was done when both the second and third authors visited the Max Plank Institute for Mathematics Bonn, in February 2017.
They thank this institution for hospitality and a fruitful working environment. F.~L. was also supported by grant CPRR160325161141 and an A-rated
scientist award both from the NRF of South Africa and by grant no. 17-02804S of the Czech Granting Agency. L.~T. was also supported by the
European Union, co-financed by the European Social Fund EFOP-3.6.1.-16-2016-00004.

\vskip4mm

\noindent
Titus Hilberdink  \\
Department of Mathematics, University of
Reading, Whiteknights,\\
PO Box 220, Reading RG6 6AX, UK\\
E-mail: {\tt t.w.hilberdink@reading.ac.uk}

\vskip2mm

\noindent
Florian Luca  \\
School of Mathematics, University of the Witwatersrand \\
Private Bag X3, WITS 2050, Johannesberg, South Africa\\
and  \\ Research Group in Algebraic Structures and Applications \\
King Abdulaziz University, Jeddah, Saudi Arabia \\
and \\
Department of Mathematics, Faculty of Sciences, University of Ostrava\\
30 dubna 22, 701 03 Ostrava 1, Czech Republic\\
E-mail: {\tt florian.luca@wits.ac.za}

\vskip2mm

\noindent
L\'aszl\'o T\'oth  \\
Department of Mathematics, University of P\'ecs \\
Ifj\'us\'ag \'utja 6, 7624 P\'ecs, Hungary \\
E-mail: {\tt ltoth@gamma.ttk.pte.hu}

\end{document}